\documentclass[reqno, letterpaper, 12pt]{amsart}

\usepackage[style=alphabetic]{biblatex}
\usepackage{multirow}
\usepackage[margin=0.75in]{geometry}
\usepackage[english]{babel}
\usepackage[utf8]{inputenc}
\usepackage{amsmath, tabu}
\usepackage{amsthm}
\usepackage{amssymb}
\usepackage{stmaryrd}
\usepackage{tikz-cd}
\usepackage{graphicx}
\usepackage[colorinlistoftodos]{todonotes}
\usepackage{enumerate}
\usepackage{hyperref}
\usepackage{extarrows}
\usepackage{parskip}
\usepackage{bm}
\usepackage{xcolor}
\usepackage{soul}
\usepackage{url}
\hypersetup{
	colorlinks,
	linkcolor={red!100!black},
	citecolor={blue!100!black},
	urlcolor={green!100!black}
}
\newtheorem{thm}{Theorem}[section]

\newtheorem{lemma}[thm]{Lemma}
\newtheorem{cor}[thm]{Corollary}

\theoremstyle{definition}
\newtheorem{defn}[thm]{Definition}
\newtheorem{hypo}[thm]{Hypothesis}
\newtheorem{ex}[thm]{Example}
\theoremstyle{remark}
\newtheorem{rmk}[thm]{Remark}

\newtheoremstyle{named}{}{}{\itshape}{}{\bfseries}{.}{.5em}{#1 \thmnote{#3}}
\theoremstyle{named}

\newtheoremstyle{named}{}{}{\itshape}{}{\bfseries}{.}{.5em}{#1 \thmnote{#3}}
\theoremstyle{named}

\newtheoremstyle{named}{}{}{\itshape}{}{\bfseries}{.}{.5em}{#1 \thmnote{#3}}
\theoremstyle{named}
\newtheorem*{namedhypo}{Hypothesis}

\newcommand{\Q}{\mathbb{Q}}
\newcommand{\Z}{\mathbb{Z}}

\newcommand{\F}{\mathbb{F}}

\newcommand{\calO}{\mathcal{O}}
\newcommand{\Sel}{\text{Sel}}
\newcommand{\corank}{\text{corank}}

\newcommand{\cyc}{\text{cyc}}
\newcommand{\Tate}{\text{Tate}}
\newcommand{\Kcyc}{K_{\mathrm{cyc}}}

\DeclareMathOperator{\Ima}{Im}
\DeclareMathOperator{\Gal}{Gal}

\newcommand{\frakM}{\mathfrak{M}}

\newcommand{\pr}{\text{pr}}

\DeclareMathOperator{\Hom}{\text{Hom}}

\DeclareMathOperator{\Ext}{Ext}

\DeclareFontFamily{U}{wncy}{}
\DeclareFontShape{U}{wncy}{m}{n}{<->wncyr10}{}
\DeclareSymbolFont{mcy}{U}{wncy}{m}{n}
\DeclareMathSymbol{\Sha}{\mathord}{mcy}{"58}

\title{\vspace{-2cm} Congruent elliptic curves over some $p$-adic Lie extensions}
\author{Tam Nguyen \& Ramdorai Sujatha}
\date{\today}

\newcommand{\calR}{\mathcal{R}}
\newcommand{\frakR}{\mathfrak{R}}
\newcommand{\ac}{\text{ac}}
\newcommand{\fine}{\text{fine}}

\newcommand{\coker}{\text{coker}}

\newcommand{\HIw}{H_{\text{Iw}}}

\newcommand{\rk}{\text{rk}}
\begin{document}
	
	\maketitle
	
	\section{Introduction}
	
	Iwasawa theory has proved to be a fruitful tool in the study of the arithmetic of elliptic curves. It is of interest therefore to delve into a deeper study of these invariants and a natural direction is to study the invariance properties of these invariants for congruent elliptic curves. Of particular interest are the $\mu$-invariants and $\lambda$-invariants associated to finitely generated torsion modules over Iwasawa algebras. Iwasawa invariants attached to modules over Iwasawa algebras that arise in this approach play a key role. Let $p$ be an odd prime number. The dual  $p^{\infty}$-Selmer groups of elliptic curves over $\Z_p$-extensions form a salient
	class whose invariants yield interesting arithmetic information. Suppose $E_1$ and $E_2$ are elliptic curves over the field $\Q$ of rational numbers such that
	$E_1[p] \simeq E_2[p]$ as $\Gal(\overline{\Q}/\Q)$-module. Such elliptic curves are called congruent, or residually isomorphic. In \cite{GV00}, Greenberg-Vatsal studied the Iwasawa theory of such elliptic curves over the cyclotomic $\Z_p$ extension under the assumption that the curves have good ordinary reduction at $p$. Under this assumption, it is a classical result due to Mazur that the dual $p^{\infty}$-Selmer groups are finitely generated modules over the corresponding Iwasawa algebra associated to the cyclotomic $\Z_p$-extension of $\Q$. Greenberg and Vatsal proved that the vanishing of the $\mu$-invariant depends only on the mod-$p$ representation. More precisely, let $\mu(E_i) ~(i=1,2)$ denote the $\mu$-invariant of the dual  $p^{\infty}$- Selmer group of the elliptic curve $E_i$ over the cyclotomic $\Z_p$-extension.
	It is proven in {\it loc.cit}  that $\mu(E_1) = 0$ implies $\mu(E_2) = 0$. Moreover, they give a formula for the difference of the $\lambda$-invariants $\lambda(E_2) - \lambda(E_1)$ in terms of explicit local constants of $E_1$ and $E_2$. Various analogues and generalizations of \cite{GV00} have been obtained in different settings (see \cite{EPW06, Kim09, PW11, CKL17, FS18}).
	
	In the context of fine Selmer groups, J. Coates and the second author conjectured that the $\mu$-invariant over the cyclotomic extension always vanishes \cite[Conjecture A]{CS05}. This conjecture is known to be invariant among residually isomorphic elliptic curves over the cyclotomic $\Z_p$-extension \cite{Suj10, FS18}. The key point in \cite{Suj10} as well as \cite{FS18} is the reformulation of Conjecture A in terms of the vanishing of the $2^{nd}$ Galois cohomology of the residual representation. It is known that the analogue of Conjecture A does not hold over general $\Z_p$-extensions. However, between two residually isomorphic elliptic curves, it is known that $\mu = 0$ for one implies the same for the other \cite[Theorem 1.1(c)]{KM22}. In this paper, we prove that the analogue of Conjecture A over a general $\Z_p$-extension is equivalent to the vanishing of the $2^{nd}$ Galois cohomology of the residual representation as in the cyclotomic case (see Section \ref{main-1}).
	
	For classical Selmer groups, we study the property $\mu = 0$ and compare $\lambda$-invariants for congruent elliptic curves over general $\Z_p$-extensions of an imaginary quadratic field $K$ in Section \ref{classical} (see theorems \ref{mu}, \ref{lambda}). This is also studied in \cite{Hac11, Kid18} using imprimitive Selmer groups in the spirit of Greenberg-Vatsal \cite{GV00}. In contrast, our techniques are based on \cite{CS23} and \cite{Ham23} where the authors define the fine residual Selmer groups to study congruences. Consider the compositum $F_\infty$ of all $\Z_p$-extensions of $K$. It is known that $F_\infty$ is a $\Z_p^2$-extension over $K$. Now, let $K_\infty/K$ be an arbitrary intermediary $\Z_p$-subextension of $F_\infty$ and let $H = \Gal(F_\infty/K_\infty)$. We give a formula for comparing $\Lambda(H)$-coranks of Selmer groups over $F_\infty$ when $\mu = 0$ over $K_\infty$ (see Theorem \ref{thm:Z_p^2}). 
	
	For an integer $m$, the False-Tate extension  $F_\infty = \Q(\mu_{p^\infty}, m^{1/p^{\infty}})$ is a rank-$2$ $p$-adic Lie extension over $\Q(\mu_p)$ whose Galois group is isomorphic to a semi-direct product $\Z_p \rtimes \Z_p$. The False-Tate extension admits $\Q(\mu_{p^\infty})$ as an intermediary $\Z_p$-subextension and we denote $H = \Gal(F_\infty/\Q(\mu_{p^\infty})).$ The classical Selmer group is known to vanish over the False-Tate extension under certain hypotheses \cite[Corollary A.11]{DDCS07}. When this occurs, the underlying elliptic curve is called regular over such an extension. In a different scenario, where the classical Selmer group acquires rank $1$ over $\Lambda(H)$, there is a systematic growth in $\Z_p$-rank of Selmer groups along the finite layers of the False-Tate extension (\cite[Theorem A.38]{DDCS07}, \cite[Proposition 4.7]{CFKS10}). In Section \ref{False-Tate}, we study how both of these behaviours vary among congruent elliptic curves by levaraging results from the work of the second author with S. Shekhar \cite{SS14} (see theorems \ref{cong:FT2} and \ref{cong:growth}).

	\section{Preliminaries} \label{prelim}
	For a number field $K$, let $K_\infty/K$ be a $\Z_p$-extension of $K$, i.e. a Galois extension where $\Gal(K_\infty/K) \simeq \Z_p$. We will denote by $K_n$ the subfield fixed by $p^n \Z_p \subset \Z_p$, so that $\Gal(K_n/K) \simeq \Z/p^n\Z$. 
	
	\subsection{Selmer groups and fine Selmer groups}
	Let $K$ be a number field and $E/K$ be an elliptic curve which is ordinary at every prime above $p$. For a finite set of primes $S$, let $K_S$ be the maximal abelian extension that is unramified outside of $S$. For every Galois extension $L/K$ contained in $K_S$, one may consider the set $S_L$ of primes in $L$ lying above $S$. For a prime $v$ in $L$, we will write $v \mid S$ to mean that $v \in S_L$. We will write $H^1(L/K, \cdot)$ in place of $H^1(\Gal(L/K), \cdot)$, and $H^1(K, \cdot)$ in place of $H^1(\Gal(\overline{K}/K), \cdot)$. We will write $L_v$ for the completion of $L$ at a place $v$, and denote by $\kappa_v$ the local Kummer map at $v$.
	\begin{defn} \label{def:Sel}
		We define the classical Selmer group $\Sel(L, E[p^\infty])$ \cite[Section 2]{Gre99} as the kernel of the natural restriction map
		\[H^1(K_S/L, E[p^\infty]) \rightarrow \bigoplus_{v \mid S, v \nmid p} H^1(L_v, E[p^\infty]) \oplus \bigoplus_{v \mid p} \frac{H^1(L_v, E[p^\infty])}{\Ima(\kappa_v)} \]
		and the fine Selmer group $\Sel_\fine(L, E)$ \cite[Section 3]{CS05}  as the kernel of
		\[H^1(K_S/L, E[p^\infty]) \rightarrow \bigoplus_{v \mid S, v \nmid p} H^1(L_v, E[p^\infty]) \oplus \bigoplus_{v \mid p} H^1(L_v, E[p^\infty]).\]
	\end{defn}
	
	For any Galois extension $L/K$, define the $p$-primary part of the Tate-Shafarevich group as
	\[\Sha(E/L)[p^\infty] = \ker(H^1(K_S/L, E)[p^\infty] \rightarrow \bigoplus_{v \mid S} H^1(L_v, E)[p^\infty]).\]
	There is a short exact sequence
	$$0 \rightarrow E(L) \otimes \Q_p/\Z_p \rightarrow \Sel(L, E[p^\infty]) \rightarrow \Sha(E/L)[p^\infty] \rightarrow 0,$$
	where $E(L) \otimes \Q_p/\Z_p \rightarrow \Sel(L, E[p^\infty])$ is induced by the Kummer map and $\Sel(L, E[p^\infty]) \rightarrow \Sha(E/L)[p^\infty]$ is induced by $H^1(K_S/L, E[p^\infty]) \rightarrow H^1(K_S/L, E)[p^\infty]$.

	Let $K_\infty/K$ be a $\Z_p$-extension. Note that
	$$\Sel(K_\infty, E[p^\infty]) = \varinjlim_{n} \Sel(K_n, E[p^\infty]),$$
	as well as
	$$\Sel_\fine(K_\infty, E[p^\infty]) = \varinjlim_{n} \Sel_\fine(K_n, E[p^\infty]).$$
	
	\subsection{Modules over the Iwasawa algebra}
	
	For a $\Z_p$-module $A$, we will denote by \[A^\vee = \Hom(A, \Q_p/\Z_p)\] to be the Pontryagin dual of $A$. The $p$-primary groups $\Sel(K_\infty, E[p^\infty]), \Sel_\fine(K_\infty, E[p^\infty])$ have the discrete topology as a $\Z_p$-module and each of $X(K_\infty, E[p^\infty]) := \Hom(\Sel(K_\infty, E[p^\infty]), \Q_p/\Z_p)$, $Y(K_\infty, E[p^\infty]) := \Sel_\fine(K_\infty, E[p^\infty])^\vee$ is a compact $\Z_p$-module. The natural actions of $G = \Gal(K_\infty/K)$ on $\Sel(K_\infty, E[p^\infty]), \Sel_{\fine}(K_\infty, E[p^\infty])$ and their duals are continuous, and these groups can be made into modules over the Iwasawa algebra \cite[Section 1]{Gre99}
	\[\Z_p \llbracket \Gal(K_\infty/K) \rrbracket := \varprojlim_n \Z_p [\Gal(K_n/K)] \simeq \Z_p \llbracket T \rrbracket.\]
	We will denote by $\Lambda(G)$  the Iwasawa algebra attached to the group $G$, and let $\Omega(G) = \Lambda(G)/p\Lambda(G) \simeq \F_p\llbracket T \rrbracket$ be its quotient modulo $p$.
	The group $G$ will often be clear from the context, in which case we may denote $\Lambda(G)$ and $\Omega(G)$ by $\Lambda$ and $\Omega$, respectively. 
	
	For a finitely-generated $\Lambda(G)$-module $X$, the most important invariants are the $\mu$- and $\lambda$-invariants, denoted $\mu_G(X)$ and $\lambda_G(X)$ respectively. We will suppress $G$ from the notation when the context is clear.
	
	\subsection{Euler characteristics}
		
	\begin{defn} \label{def:Euler-char}
		Let $G$ be a topological group, and $M$ be a topological $G$-module on which $G$ acts continuously. Suppose that $H^i(G, M)$ is finite for each $i \geq 0$ and $H^i(G, M) = 0$ for $i$ sufficiently large. We define the Euler characteristic $\chi(G, M)$ via the formula $$\chi(G, M) := \prod_{i \geq 0} (\# H^i(G, M))^{(-1)^{i}}.$$ 
	\end{defn}
	\subsection{Hypotheses}
	Throughout this article, we assume that $E/K$ has good ordinary reduction at every prime above $p$:
	
	\begin{namedhypo}[(ord)] \label{ord}
		$E/K$ has good ordinary reduction at every prime above $p$. 
	\end{namedhypo}

	We also define the following hypotheses for an elliptic curve $E/K$:

	\begin{namedhypo}[(weak-Leop)] \label{weak-Leop}
		$H^2(K_S/K_\infty, E[p^\infty]) = 0$. 
	\end{namedhypo}
	
	\begin{namedhypo}[(van)] \label{van}
		$E(K_v)[p] = 0$ for every prime $v \in S$ that is infinitely split in the extension $K_\infty/K$. 
	\end{namedhypo}

	\begin{namedhypo}[(cotor)] \label{cotor}
		$\Sel(K_\infty, E[p^\infty])$ is $\Lambda(G)$-cotorsion.
	\end{namedhypo}

	\begin{namedhypo}[(surj)] \label{surj}
		$E$ does not admit complex multiplication and the natural map $\Gal(K(E[p^\infty])/ K) \rightarrow GL_2(\Z_p)$ is surjective.  
	\end{namedhypo}

	\begin{namedhypo}[($p$-ram)] \label{p-ram}
		Every prime of $K$ above $p$ ramifies in $K_\infty/K$.
	\end{namedhypo}

	\section{Fine Selmer groups over general $\Z_p$-extensions}  \label{fineSelmer}
	In \cite{FS18}, the authors study fine Selmer groups for congruent elliptic curves over certain $p$-adic Lie extensions containing the cyclotomic extension. Such a result has been extended to a wider class of $p$-adic Lie extensions (which do not necessarily contain the cyclotomic one) in \cite{KM22}. We stress that our arguments are based on \cite{Suj10}. Moreover, the properties in this section are studied in connection with classical Selmer groups which are discussed in Section \ref{classical}.

	In this section, let $K$ be a number field, $K_\infty = \cup_{n \geq 0} K_n$ be an arbitrary $\Z_p$-extension over $K$ and $G = \Gal(K_\infty/K)$. Let $E/K$ be an elliptic curve where \nameref{weak-Leop} holds.

	\begin{thm} \label{thm:weak-Leop}
		Assume that the $K_\infty/K$ satisfies \nameref{p-ram}, and $E/K$ satisfies hypothesis \nameref{cotor}. Then $E/K$ also satisfies \nameref{weak-Leop}.
	\end{thm}

	\begin{proof}
		We will show that $H^2(K_S/K_{\infty}, E[p^\infty]) = 0$ using corank computations and the following defining sequence:
		\[0 \rightarrow \Sel(K_{\infty}, E[p^\infty]) \rightarrow H^1(K_S/K_{\infty}, E[p^\infty]) \rightarrow \bigoplus_{v \in S} J_v^{1}(K_{\infty}, E[p^\infty])\]
		where we denote $J_v^{1}(K_\infty, E[p^\infty]) = \bigoplus_{w \mid v} H^1(K_{\infty, w}, E[p^\infty])$ for a place $v$ in $K$ such that $v \nmid p$, and $J^{1}_v(K_\infty, E[p^\infty]) = \bigoplus_{w \mid v} \frac{H^1(K_w, E[p^\infty])}{\Ima(\kappa_w)}$ for $v \mid p$.
		
		Indeed, for $v \nmid p$, $\corank_{\Lambda(G)} J_v^{1}(K_\infty, E[p^\infty]) = \corank_{\Lambda(G_v)} H^1(K_{\infty, w}, E[p^\infty]) = 0$ for any choice of $w \mid v$ by \cite[Proposition 2]{Gr89}.
		For $v \mid p$, because $v$ is ramified in $K_\infty/K$ by \nameref{p-ram}, one can apply \cite[Theorem 4.9]{CG96} to obtain $\corank_{\Lambda} \bigoplus_{v \mid p} J_v^{1} (K_{\infty}, E[p^\infty]) = [K:\Q]$. Combining these calculations gives $\corank_{\Lambda} \bigoplus_{v \in S} J_v^{1} (K_{\infty}, E[p^\infty]) = [K:\Q]$. 
		
		It then follows from the defining exact sequence above that $\corank_{\Lambda} H^1(K_S/K_{\infty}, E[p^\infty]) \leq [K:\Q]$. Moreover, \cite[Theorem 3.1.4]{P-R92} gives $\corank_{\Lambda} H^1(K_S/K_{\infty}, E[p^\infty]) \geq [K:\Q]$ and hence the equality must hold. 
		
		We have the following formula (\cite[Proposition 3]{Gr89}):
		\[\corank_{\Lambda} H^1(K_S/K_\infty, E[p^\infty]) - \corank_{\Lambda} H^2(K_S/K_\infty, E[p^\infty]) \geq [K: \Q],\]
		because the $p$-adic representation induced by $E$ is odd (see \cite[p.12]{Kid18} for a similar argument).
		It then follows that $\corank_{\Lambda}{H^2(K_S/K_{\infty}, E[p^\infty])} = 0$.
		
		Moreover, \cite[Proposition 4]{Gr89} states that $H^2(K_S/K_{\infty}, E[p^\infty])$ is $\Lambda$-cofree. It immediately follows that $H^2(K_S/K_{\infty}, E[p^\infty]) = 0$.
	\end{proof}
	\begin{rmk}
		We remark that the hypothesis \nameref{p-ram} is satisfied for the cyclotomic $\Z_p$-extension $K_\cyc/K$ since $K_\cyc = K \cdot \Q_\cyc$ and $\Q_\cyc/\Q$ is ramified at $p$, and in fact totally ramified. However, it does not always hold for a general $\Z_p$-extension $K_\infty/K$.
	\end{rmk}
	We give the following example for Theorem \ref{weak-Leop}.
	
	\begin{ex}
		Suppose that $E/\Q$ is an elliptic curve which is ordinary at $p$ whose discriminant $N$ can be written as $N =N^{-} N^{+}$ where $N^{+}$ (resp $N^{-}$) is divisible by the primes which are split (resp. inert) in $K$. Let $K$ be an imaginary quadratic field and $K_\infty = K_\ac$, the anticyclotomic $\Z_p$-extension of $K$. For a fixed elliptic curve $E$, it is a result of Bertolini-Darmon \cite{BD05} that $\Sel(K_\ac, E[p^\infty])$ is $\Lambda(G)$-cotorsion for all but finitely many primes $p$. Suppose also that $p$ is inert in $K$, which implies that $p$ ramifies in $K_\infty/K$ by \cite[Lemma 13.3]{Was97}. Then the theorem above implies \nameref{weak-Leop} for $E/K$.
	\end{ex}

	\begin{defn} \label{Iw-coh}
		For a compact $\Z_p$-module $M$ with a continuous $\Gal(K_\infty/K)$-action, define the Iwasawa cohomology to be
		$$\HIw^i(K_\infty, M) = \varprojlim_{n} H^i(K_S/K_n, M)$$
		with respect to corestriction maps.
	\end{defn}

	We state the following extension of \cite[Proposition 4.6]{Suj10}, which is the main theorem of this chapter.
	
	\begin{thm} \label{main-1}
		Assume that $E/K$ satisfies \nameref{weak-Leop} and \nameref{van}. 
		Then the fine Selmer group $Y(K_\infty, E[p^\infty])$ is finitely generated over $\Z_p$ (or, equivalently, has trivial $\mu$-invariant) if and only if $H^2(K_S/K_\infty, E[p]) = 0$. 
	\end{thm}

	We state the following observation as a corollary.
	\begin{cor} \label{cor:fineSel}
		Let $E_1/K, E_2/K$ be elliptic curves such that $E_1[p] \simeq E_2[p]$ as $\Gal(\overline{K}/K)$-modules and that \nameref{weak-Leop} holds for both $E_1$ and $E_2$, i.e. $H^2(K_S/K_\infty, E_i[p^\infty])= 0$ for $i = 1, 2$. Then $\mu(Y(K_\infty, E_1[p^\infty])) = 0$ if and only if $\mu(Y(K_\infty, E_2[p^\infty])) = 0.$
	\end{cor}
	
	As in the previous chapter, let $S$ be the finite set of primes in $K$ containing the primes of bad reduction of $E/K$, the primes above $p$ and the primes at infinity. Following the notations of \cite[Lemma 3.1]{CS05}, for a prime $v \in S$, let $K_v^{i}(K_\infty, E[p^\infty]) = \bigoplus_{w \mid v} H^i(K_{\infty, w}, E[p^\infty])$, which are the local cohomology groups of the defining sequence of the fine Selmer group. Moreover, let $U_v$ and $A_v$ respectively be the Pontryagin duals of $K_v^0(K_\infty, E[p^\infty])$ and $H^0(K_{\infty, w}, E[p^\infty])$ for some fixed prime $w$ of $K_\infty$ above $v$. In general, we have $U_v = \Hom_{\Lambda(G_v)}(\Lambda(G), A_v)$ because of the isomorphism $K_v^0(K_\infty, E[p^\infty]) \simeq \Lambda(G) \otimes_{\Lambda(G_v)} H^0(K_{\infty, w}, E[p^\infty])$ and 
	the tensor-hom adjuction. When $v$ is finitely split in $K_\infty/K$, we have $U_v = \Hom_{\Lambda(G_v)}(\Lambda(G), A_v) = \Lambda(G) \otimes_{\Lambda(G_v)} A_v$ because $[G: G_v] < \infty.$

	To obtain Theorem \ref{main-1}, we will use the following lemmas.
	\begin{lemma} \label{equal-mu}
		Suppose that $E/K$ satisfies hypothesis \nameref{van}. Then the $\Lambda$-modules $\HIw^2(K_\infty, T_p(E))$ and $Y(K_\infty, E[p^\infty])$ have the same $\mu$-invariant.
	\end{lemma}
	\begin{proof}
		The Poitou-Tate exact sequence gives:
		$$0 \rightarrow H^0(K_\infty, E[p^\infty]) \rightarrow \bigoplus_{v \in S} K_v^{0}(K_\infty, E[p^\infty]) \rightarrow \HIw^2(K_\infty, T_p(E))^\vee \rightarrow Y(K_\infty, E[p^\infty])^\vee \rightarrow 0.$$
		
		The surjection follows immediately from the fact that $Y(K_\infty, E[p^\infty])^\vee = \Sel_{\fine}(K_\infty, E[p^\infty])$ is defined as the kernel of $H^1(K_S/K_\infty) \rightarrow \bigoplus_{v \in S} K^1_v(K_\infty, E[p^\infty]).$
		
		Since $H^0(K_\infty, E[p^\infty])^\vee$ is finitely generated over $\Z_p$, it suffices to show that $\mu_G(U_v) = 0$ for every $v \in S$. For primes $v$ in $K$ that are infinitely split in $K_\infty/K$, we assumed that $E/K$ satisifes hypothesis \nameref{van}, hence $U_v = 0$.
		
		Suppose that $v$ is finitely split in $K_\infty/K$. Because $\Lambda(G)$ is flat over $\Lambda(G_v)$ \cite[821]{CS05}, it can be shown that $U_v(p) = \Lambda(G) \otimes_{\Lambda(G_v)} A_v(p)$. Indeed, the following sequence must be exact:
		$$0 \rightarrow A_v[p^n] \otimes_{\Lambda(G_v)} \Lambda(G) \rightarrow A_v \otimes_{\Lambda(G_v)} \Lambda(G) \rightarrow A_v \otimes_{\Lambda(G_v)} \Lambda(G) \rightarrow 0,$$
		
		hence $U_v[p^n] \simeq A_v[p^n]\ \otimes_{\Lambda(G_v)} \Lambda(G)$.  Taking direct limits then gives $U_v(p) = \Lambda(G) \otimes_{\Lambda(G_v)} A_v(p)$.

		By Shapiro's lemma, we also have$$\chi(G, U_v(p)) = \chi(G_v, A_v(p))$$
		
		where $\chi(G, -), \chi(G_v, -)$ are the Euler characteristics. In the following argument, we exploit the fact that $\chi(G, X(p)) = p^{\mu_G(X)}$ for a finitely generated $\Lambda(G)$-torsion module $X$ (see \cite[page 811]{CS05}). Now $\mu_{G_v}(A_v(p)) = 1$ because $A_v$ is finitely generated and torsion over $\Lambda({G_v})$ and $A_v(p)$ is finite (note that $A_v$ is finitely generated over $\Z_p$). Hence $U_v$ is finitely generated and torsion over $\Lambda(G)$, and $\chi(G, U_v(p)) = \chi(G_v, A_v(p)) = 1$, which implies $\mu_G(U_v) = 0$. 
	\end{proof}
%
	\begin{rmk}
		When \nameref{van} does not hold for some infinitely split prime $v$, we claim that $\mu_G(U_v) > 0$. Indeed, we first note that the decomposition group $G_v$ is trivial. Moreover, $U_v$ and $A_v$ are $p$-primary, so that $A_v(p) = A_v$ and $U_v(p) = U_v$. We have the following:
		\[\chi(G_v, A_v(p)) = \chi(G_v, A_v) = \# A_v =  \# E[p^\infty](K_v).\]
		It follows that $\chi(G, U_v(p)) = \chi(G, U_v) = \chi(G_v, A_v) = \# E[p^\infty](K_v)$ by Shapiro's lemma. If \nameref{van} does not hold for $v$, then $\mu_G(U_v) > 0$ because $\chi(G, U_v(p)) = p^{\mu_G(U_v(p))}$ (see \cite[page 811]{CS05}).
	\end{rmk}

	\begin{lemma} \label{equiv-fine} Each of the following is equivalent to \nameref{weak-Leop}:
		\begin{enumerate}[(i)]
			\item $Y(K_\infty, E[p^\infty])$ is $\Lambda(G)$-torsion.
			\item $\HIw^2(K_\infty, T_p(E))$ is $\Lambda(G)$-torsion.
		\end{enumerate}
	\end{lemma}
	\begin{proof}
		The proof for \cite[Lemma 3.1]{CS05} in the cyclotomic setting also applies more generally in our setting, the only difference being the presence of primes $v \in S$ that are infinitely split in $K_\infty/K$.
		
		Observe that for a bad prime $v$ of $K$, $G_v \simeq \Z_p$ if $v$ is finitely decomposed, and $G_v \simeq 1$ if $v$ is infinitely split. In the first case,  $G_v$ has dimension $1$ as a $p$-adic Lie group and hence $A_v$ is $\Lambda(G_v)$-torsion as a finitely generated $\Z_p$-module. In the case that $G_v \simeq 1$, $v$ is totally and infinitely split in $K_\infty/K$. Let $w$ be a prime in $K_\infty$ above $v$. Then we have $K_{\infty, w} = K_v$ and $A_v := E[p^\infty](K_{\infty, w}) = E[p^\infty](K_v)$, which is finite and therefore $\Lambda(G_v)$-torsion. 
		
		The rest of the proof is analogous to the proof of \cite[Lemma 3.1]{CS05}. 	
	\end{proof}
	
	\begin{lemma}\label{Iw-Gal}
		Assume that $E/K$ satsifies \nameref{weak-Leop}. Then the $\Omega(G)$-module $\HIw^2(K_\infty, E[p])$ is finite if and only if $H^2(K_S/K_\infty, E[p]) = 0$.
	\end{lemma}
	\begin{proof}
		By Jannsen's spectral sequence \cite[Theorem 1]{Jan14} and the fact that $T_p(E[p]) = E[p]$, one has $$E_2^{p,q} = {E}^p(H^q(K_S/K_\infty, E[p])^\vee) \Rightarrow  \HIw^{p + q}(K_S/K_\infty, E[p])$$ 
		where $E^i(-) := \Ext_{\Omega(G)}^i(-, \Omega(G))$.
		If $H^2(K_S/K_\infty, E[p]) = 0$ then the following is exact:
		$$0 \rightarrow E^2(X_0) \rightarrow \HIw^2(K_\infty, E[p]) \rightarrow E^1(X_1)$$
		where $X_i = H^i(K_S/K_\infty, E[p])^{\vee}$. 
		
		Both $E^2(X_0)$ and $E^1(X_1)$ are $\Omega(G)$-torsion as $X_0$ and $X_1$ are finitely generated as $\Omega(G)$-modules, and $\Omega(G)$ is a PID. Hence $\HIw^2(K_\infty, E[p])$ is also $\Omega(G)$-torsion, which is equivalent to being finite as an $\Omega(G)$-module. Conversely, if $\HIw^2(K_\infty, E[p])$ is finite then one can extract the following exact sequence from Jannsen's spectral sequence:
		\[\HIw^2(K_\infty, E[p]) \rightarrow E^0(X_2) \rightarrow E^2(X_1)\]
		
		As before, $E^2(X_1)$ is $\Omega(G)$-torsion and so is $E^0(X_2) := \Hom_{\Omega(G)}(X_2, \Omega(G))$. Therefore, $X_2$ is also torsion. The short exact sequence $0 \rightarrow E[p] \rightarrow E[p^\infty] \rightarrow E[p^\infty] \rightarrow 0$ induces the following sequence in cohomology:
		$$H^2(K_S/K_\infty, E[p^\infty])^{\vee} \rightarrow H^2(K_S/K_\infty, E[p])^{\vee} \rightarrow H^1(K_S/K_\infty, E[p^\infty])^\vee \rightarrow H^1(K_S/K_\infty, E[p^\infty])^\vee.$$
		
		By \cite[Proposition 5]{Gr89}, $H^2(K_S/K_\infty, E[p^\infty]) = 0$ implies that $H^1(K_S/K_\infty, E[p^\infty])^{\vee}$ has no nontrivial finite submodule. It is then clear form the exact sequence above that $H^2(F_S/K_\infty, E[p]) = 0$.
	\end{proof}
	
	We are now ready to prove the main theorem of this chapter.
	\begin{proof}[Proof of Theorem \ref{main-1}]
		Consider the Iwasawa cohomologies attached to $T_p(E)$ and $E[p]$:
		$$\HIw^2(K_\infty, T_p(E)) := \varprojlim_{n} H^2(K_S/ K_n, T_p(E)), \HIw^2 (K_\infty, E[p]) := \varprojlim_{n} H^2(K_S/K_n, E[p]).$$
		Then by Lemma \ref{Iw-Gal}, $\HIw^2(K_\infty, E[p])$ is finite if and only if $H^2(K_S/K_\infty, E[p]) = 0$. 
		
		From Lemma \ref{equal-mu} and \ref{equiv-fine}, it follows that both $Y(K_\infty, E[p^\infty])$ and $\HIw^2(K_\infty, T_p(E))$ are $\Lambda(G)$-torsion and that they have the same $\mu$-invariant. Since $\HIw^2(K_\infty, E[p]) \simeq \HIw^2(K_\infty, T_p(E))/p$, it follows that $\HIw^2(K_\infty, E[p])$ is finite if and only if $\HIw^2(K_\infty, T_p(E))$ is finitely generated over $\Z_p$. This completes the proof of the theorem.
	\end{proof}

	\section{Classical Selmer groups over general $\Z_p$-extensions}
	\label{classical}
	In \cite{Hac11, Kid18}, the authors study how the $\mu$ and $\lambda$-invariants vary for general $\Z_p$-extensions over number fields. An important hypothesis in those papers is that the Selmer group over the trivializing extension is cotorsion over the Iwasawa algebra.  As remarked in \cite[303, 304]{Kid18}, when the base field is imaginary quadratic, this hypothesis is known to be true over the cyclotomic extension and the anticyclotomic extension in the definite setting. It is now known to hold for all but finitely many $\Z_p$-extensions \cite[Proposition 2.5, Proposition 2.6]{KMR23}. 

	In this section, let $K$ be an imaginary quadratic field of discriminant $-D < 0$. Let $E/K$ be an elliptic curve and suppose that $p > 2$ is a prime such that $D, N$ and $p$ are pairwise coprime. We will denote by $K_{\ac}$ and $K_{\cyc}$ the anticyclotomic and cyclotomic extensions of $K$, respectively. In the special cases where we consider these extensions, we will denote $G_{\star} = \Gal(K_{\star}/K)$ and denote $\Lambda(G_{\star})$ to be the Iwasawa algebra attached to $G_{\star}$, where $\star \in \{\ac, \cyc\}$. Again, let $S$ be a finite set of primes in $K$ containing the primes of bad reduction of $E/K$, and the primes above $p$ and the primes at infinity.
	
	In order to state the main results, we recall the following hypotheses for our imaginary quadratic field $K$, which were stated in the introduction for general number fields.
	
	\begin{namedhypo}[\nameref{ord}]
		$E/K$ has good ordinary reduction at every prime above $p$.
	\end{namedhypo}
	
	\begin{namedhypo} [\nameref{cotor}] 
		$\Sel(K_\infty, E[p^\infty])$ is $\Lambda(G)$-cotorsion. 
	\end{namedhypo}

	\begin{namedhypo} [\nameref{surj}] 
		$E$ does not admit complex multiplication and the natural map \[\Gal(K(E[p^\infty])/ K) \rightarrow GL_2(\Z_p)\] is surjective.  
	\end{namedhypo}
	
	\begin{rmk} \label{rmk:surj}
		Hypothesis \nameref{surj} implies that $H^0(K_S/K_n, E[p^\infty]) = 0$ for every finite layer $K_n/K$ of $K_\infty/K$. Indeed, if $P \in E[p^\infty]$ is of exact order $p^m$ and $P \in K_n$, then $\{P^\sigma \mid \sigma \in \Gal(K(E[p^m])/K)\} \subset K_n$ and so $K(E[p^m]) \subset K_n$ because the Galois action is transitive. However, $\Gal(K_n/K)$ cannot admit $\Gal(K(E[p^m])/K)$ as a $GL_2(\Z/p^m\Z)$-quotient. This is because $K_n/K$ has degree $p^n$, and $GL_2(\Z/p^m \Z)$ contains the subgroup $(\Z/p^m\Z)^\times$ of order $\phi(p^m) = p^{m - 1}(p - 1)$. 
	\end{rmk}
	

	\begin{namedhypo}[\nameref{p-ram}] 
		Every prime of $K$ above $p$ ramifies in $K_\infty/K$.
	\end{namedhypo}
	
	\begin{namedhypo}[\nameref{van}] 
		$E(K_v)[p] = 0$ for every prime $v \in S$ that is infinitely split in the extension $K_\infty/K$. 
	\end{namedhypo} 
	
	Note that for a fixed elliptic curve $E$, there are only finitely many primes $p$ that do not satisfy this assumption. This is because the condition that $E(K_v)[p] = 0$ is equivalent to $p \nmid 1 + q_v - a_v$. Here, we denote $a_v := 1 + q_v - \# \tilde{E}_v(k_{v})$, $\tilde{E}_v$ is the reduction of $E$ at $v$ and $q_{v}$ is the size of the residue field $k_v$ of $K_v$.

	Let $E_1/\Q$ and $E_2/\Q$ be a pair of $p$-residually isomorphic elliptic curves of conductors $N_1$ and $N_2$, respectively. It follows from a result of Kleine, Matar and Ramdorai \cite[Proposition 2.5]{KMR23} that for all but finitely many $\Z_p$-extensions $K_\infty/K$, the following properties are satisfied:
	
	\begin{enumerate}
		\item No prime dividing $pN_1 N_2 \infty$ splits completely in $K_\infty/K$. 
		\item Every prime of $K$ above $p$ ramifies in $K_\infty/K$.
		\item $\Sel(K_\infty, E_i[p^\infty])$ is $\Lambda(G)$-cotorsion for each $i$.
	\end{enumerate}
	
	Indeed, observe that the hypotheses \nameref{cotor}, \nameref{van}, \nameref{p-ram} are satisfied for the cyclotomic extension $K_\cyc$. Moreover, \cite[Proposition 2.5]{KMR23} only assumes the existence of one extension $K_\infty/K$ that satisfies all three properties. Hence, one may take $K_\infty$ to be the cyclotomic extension in \cite[Proposition 2.5]{KMR23} and conclude that \nameref{cotor}, \nameref{van}, \nameref{p-ram} are satisfied for all but finitely many other $\Z_p$-extensions.
	
	
	\begin{rmk}
		Consider the special case where $K_\infty = K_{\ac}$, the anticyclotomic extension of $K$. One may factorize $N = N^{+} N^{-}$ with $N^+$ divisible by primes that split in $K/\Q$ and $N^{-}$ divisible by inert primes. Now suppose that  $N^{-}$ has an odd number of prime divisors. For all but finitely many primes $p$, it is a theorem that $\Sel(K_\ac, E[p^\infty])$ is cotorsion (see \cite[Introduction]{BD05}). Pollack and Weston later made improvements for Galois representations associated with modular forms \cite{PW11}. For a fixed elliptic curve, there are only finitely many primes $p$ that do not satisfy all of the hypotheses \nameref{ord}, \nameref{cotor}, \nameref{surj}, \nameref{van} and \nameref{p-ram}. In this setting, the results in this section are special cases of \cite[Theorem 7.1]{PW11} if one only restricts to elliptic curves. However, the arguments we use are based on \cite{CS23}, which examines an analogous question for $\pm$-Selmer groups in the cyclotomic supersingular case using fine residual Selmer groups.	
	\end{rmk}

	
	We are now ready to state the first main result of this section.
	\begin{thm} \label{mu}
		Assume that $E_1/K$ and $E_2/K$ are elliptic curves such that $E_1[p] \simeq E_2[p]$ as $\Gal(\overline{K}/K)$-modules and that they both satisfy hypotheses \nameref{ord}, \nameref{cotor}, \nameref{surj}, \nameref{van} and \nameref{p-ram}. Then $\mu(\Sel(K_\infty, E_1[p^\infty])) = 0$ if and only if $\mu(\Sel(K_\infty, E_2[p^\infty])) = 0$. 
	\end{thm}

	
	We begin with the following definitions:
	
	\begin{defn} 
		We define the fine residual Selmer group $R(K_n, E[p])$ to be the kernel of
		$$H^1(K_S/K_n, E[p]) \rightarrow \bigoplus_{v \in S, v \nmid p} H^1(K_{n, v}, E[p]) \oplus \bigoplus_{v \mid p} H^1(K_{n, v}, \tilde{E}_v[p]),$$
		where $\tilde{E}_v$ is the reduction modulo $v$ of $E/K_v$.
	\end{defn}
	
	Moreover, let  $R(K_{\infty}, E[p]) = \varinjlim_{n} R(K_n, E[p])$. For a prime $v$ of $K_n$, put $$W_v = \begin{cases}
		0 & \text{ if } v \mid S \text{ and } v \nmid p, \\
		\ker(H^1(K_{n, v}, E[p]) \rightarrow H^1(K_{n, v}, \tilde{E}_v[p])) & \text{ if } v \mid p.
	\end{cases}$$
	Then the fine residual Selmer group $R(K_n, E[p])$ is the kernel of
	\[H^1(K_S/K_n, E[p]) \rightarrow \bigoplus_{v \mid S} H^1(K_{n, v}, E[p])/ W_v.\]
	
	Let $W_v^\perp$ be the orthogonal complement of $W_v$ under the local Tate pairing
	$$H^1(K_{n, v}, E[p]) \times H^1(K_{n, v}, E[p]) \rightarrow \Q_p/\Z_p.$$
	
	\begin{defn}
		Define $\frakR(K_n, E[p])$ as the kernel of
		$$H^1(K_S/K_n, E[p]) \rightarrow \bigoplus_{v \mid S} H^1(K_{n, v}, E[p])/ W_v^{\perp}.$$
	\end{defn}
	
	\begin{defn}
		We define
		$\calR(K_\infty, E[p]) = \varprojlim_{n} \frakR(K_n, E[p])$ with respect to the corestriction maps $\frakR(K_{n + 1}, E[p]) \rightarrow \frakR(K_n, E[p]).$
	\end{defn}
	
	Observe that $\calR(K_\infty, E[p])$ is a submodule of the Iwasawa cohomology \[\HIw^1(K_\infty, E[p]) := \varprojlim_{n} H^1(K_n, E[p])\] in Definition \ref{Iw-coh}. The proof of the following lemma is completely analogous to \cite[Lemma 2.7]{CS23}:
	\begin{lemma} \label{free}
		Assume hypothesis \nameref{surj}. Then $\calR(K_\infty, E[p])$ is free over $\Omega(G)$.
	\end{lemma}
	
	\begin{proof}
		We recall the Jannsen's spectral sequence \cite{Jan14}:
		$$E_2^{p,q} = {E}^p(H^q(K_S/K_\infty, E[p])^\vee) \Rightarrow  \HIw^{p + q}(K_S/K_\infty, E[p]),$$ 
		where $E^i(-) := \Ext_{\Omega(G)}^i(-, \Omega(G)).$
		Hypothesis \nameref{surj} gives $E_2^{p, 0} = 0$ for all $p \geq 0$. For $j \geq 2$, $E_{j + 1}^{1, 0}$ is the cohomology at $E_j^{1, 0}$ of the sequence $0 \rightarrow E_j^{1, 0} \rightarrow 0$, hence we have $E_\infty^{1, 0} = E_2^{1, 0} = 0$. Since $E_2^{2, 0} = 0$, and for $j \geq 2$, $E_{j + 1}^{0, 1}$ is the cohomology taken at $E_j^{0, 1}$ of the sequence $0 \rightarrow E_2^{0, 1} \rightarrow E_2^{2, 0} \rightarrow 0$, we also have $E_\infty^{ 0, 1} = E_2^{0, 1}$. 
		
		Because $E_\infty^{1, 0}$ embeds in $\HIw^1(K_S/K_\infty, E[p])$ with quotient $E^{0, 1}_\infty$, one has an isomorphism $$\Hom(H^1(K_S/K_\infty, E[p]^\vee), \Omega(G)) = \Ext^{0}_{\Omega(G)}(H^1(K_S/K_\infty, E[p]^\vee)) \simeq \HIw^1(K_S/K_\infty, E[p]).$$ 
		Hence, $\HIw^1(K_S/K_\infty, E[p])$ is free over $\Omega(G)$.  
		By taking inverse limits of the inclusions $\frakR(K_n, E[p]) \hookrightarrow H^1(K_S/K_n, E[p]),$ we have an injection $\calR(K_\infty, E[p]) \hookrightarrow \HIw^1(K_S/K_\infty, E[p])$. Because $\Omega(G)$ is a principal ideal domain and $\HIw^1(K_S/K_\infty, E[p])$ is free over $\Omega(G)$, $\calR(K_\infty, E[p])$ is also free over $\Omega(G)$. 
	\end{proof}	
	
	\begin{defn}
		We define 
		$$\tilde{K}_v^{1}(K_\infty, E[p]) = \bigoplus_{w \mid v} H^1(K_{\infty, w}, \tilde{E}_v[p])$$ 
		if $v \mid p$ and
		$$\tilde{K}_v^{1}(K_\infty, E[p]) = \bigoplus_{w \mid v} H^1(K_{\infty, w}, E[p])$$
		if $v \in S$ and $v \nmid p$. Also, let $\xi_p$ be the following map in the defining sequence of $R(K_\infty, E[p])$:
		$$0\rightarrow R(K_\infty, E[p]) \rightarrow H^1(F_S/K_\infty, E[p]) \xrightarrow{\xi_p} \bigoplus_{v \in S} \tilde{K}_v^{1}(K_\infty, E[p]).$$
	\end{defn}
	
	We state the following analogue of \cite[Lemma 3.7(i)]{CS23} for our setting:
	\begin{thm} \label{ker}
		\begin{enumerate}[(i)] 
			
			\item Assuming hypothesis \nameref{surj}, the map $$H^1(K_S/K_\infty, E[p]) \rightarrow H^1(K_S/K_\infty, E[p^\infty])[p]$$ induced by $E[p] \hookrightarrow E[p^\infty]$ is an isomorphism.
			\item For every prime $w \mid S$ where $w \nmid p$, the map
			$$\psi_{w}: H^1(K_{\infty, w}, E[p]) \twoheadrightarrow H^1(K_{\infty, w}, E[p^\infty])[p]$$ induced by $E[p] \hookrightarrow E[p^\infty]$
			is surjective with finite kernel, of dimension $\dim_{\F_p}(\ker \psi_{w}) = \dim_{\F_p} E(K_{\infty, w})[p] \leq 2$.
			
			\item Assume hypothesis \nameref{p-ram}. For every prime $w \mid v$, where $v \mid p$ is a prime in $K$, the morphism
			\[\widetilde{\psi_w}: H^1(K_{\infty, w}, \tilde{E}_v[p]) \twoheadrightarrow H^1(K_{\infty, w}, E)[p] = (H^1(K_{\infty, w}, E[p^\infty])/ \Ima(\kappa_{\infty, w}))[p]\] 
			is surjective with finite kernel, of dimension $\dim_{\F_p}(\ker{\widetilde{\psi_w}}) = \dim_{\F_p} \tilde{E}_v(k_{\infty, w})[p] \leq 1$ where $k_{\infty, w}$ is the residue field of $K_{\infty, w}$, and $\kappa_{\infty, w}: E(K_{\infty, w}) \otimes \Q_p/\Z_p \rightarrow H^1(K_{\infty, w}, E[p^\infty])$ is the Kummer map. 
		\end{enumerate}
	\end{thm}
	
	\begin{proof}
		\begin{enumerate}[(i)]
			\item The short exact sequence $0 \rightarrow E[p] \rightarrow E[p^\infty] \xrightarrow{\times p} E[p^\infty] \rightarrow 0$ induces the following exact sequence:
			$$0 \rightarrow H^0(K_S/K_n, E[p^\infty])/p H^0(K_S/K_n, E[p^\infty])  \rightarrow H^1(K_S/K_n, E[p]) \rightarrow H^1(K_S/K_n, E[p^\infty])[p] \rightarrow 0$$ 
			Note that $H^0(K_S/K_n, E[p^\infty]) = 0$, based on hypothesis \nameref{surj}.
			Now, by taking direct limits of isomorphisms
			$$H^1(K_S/K_n, E[p]) \xrightarrow{\simeq} H^1(K_S/K_n, E[p^\infty])[p]$$
			one obtains 
			$$H^1(K_S/K_\infty, E[p]) \xrightarrow{\simeq} H^1(K_S/K_\infty, E[p^\infty])[p]$$
			\item The proof follows from the proof of \cite[Proposition 4.1(b)]{CS23}. Indeed, one also has a local short exact sequence:
			$$0 \rightarrow H^0(K_{n, v_n}, E[p^\infty])/p H^0(K_{n, v_n}, E[p^\infty])  \rightarrow H^1(K_{n, v_n}, E[p]) \xrightarrow{\psi_n} H^1(K_{n, v_n}, E[p^\infty])[p] \rightarrow 0$$
			where $H^0(K_{n, v_n}, E[p^\infty])/p H^0(K_{n, v_n}, E[p^\infty]) \simeq E(K_{n, v_n})[p]$. Since taking kernels commutes with taking direct limits, we have $\ker(\psi_w) = \varinjlim_{n} \ker(\psi_{v_n}) = E(K_{\infty, w})[p]$, and the result immediately follows.
			
			\item The proof follows from the proof of \cite[Proposition 4.1(c)]{CS23}. Note that Hypothesis \nameref{p-ram} implies that our $\Z_p$-extension is locally deeply ramified at every prime $v \mid p$ in $K$, in the sense of \cite{CG96}.
			
			
			Consider the exact sequence
			$$0 \rightarrow \widehat{E}_v[p^\infty] \rightarrow E[p^\infty] \rightarrow \tilde{E}_v[p^\infty] \rightarrow 0$$
			where $\widehat{E}_v$ is the formal group of $E$ over the ring of integers $\calO_{K_{v}}$ of $K_{v}$. There is an induced exact sequence in cohomology:
			
			\begin{align*}
				0 \rightarrow H^1(K_{\infty, w}, E[p^\infty]) / \Ima(H^1(K_{\infty, w}, E[p^\infty])) \rightarrow H^1(K_{\infty, w}, \tilde{E}_v[p^\infty]) \\ \rightarrow H^2(K_{\infty, w}, \widehat{E}_v[p^\infty]).
			\end{align*}
			
			By the proof of \cite[Proposition 9, Chapter 2, \S 3.3]{Ser97}, $\Gal(\overline{K_v}/K_{\infty, w})$ has $p$-cohomological dimension $1$ since $K_{\infty, w}/K_v$ is a $\Z_p$-extension by Hypothesis \nameref{p-ram}. Therefore, we have $H^2(K_{\infty, w}, \widehat{E}_v[p^\infty]) = \varinjlim_{n} H^2(K_{\infty, w}, \widehat{E}_v[p^n]) = 0$.
			
			From the short exact sequence above, we have an isomorphism 
			\begin{align*}
				H^1(K_{\infty, w}, E)[p^\infty] \simeq H^1(K_{\infty, w}, E[p^\infty])/\Ima(\kappa_{\infty, w}) \simeq H^1(K_{\infty, w}, E[p^\infty])/\Ima(H^1(K_{\infty, w}, \widehat{E}_v[p^\infty])) \\ \simeq H^1(K_{\infty, w}, \tilde{E}_v[p^\infty]),
			\end{align*}
			which follows from the identifications $H^1(K_{\infty, w}, E)[p^\infty] \simeq H^1(K_{\infty, w}, E[p^\infty])/\Ima(\kappa_{\infty, w})$ \cite[p.62]{Gre99} and $\Ima(\kappa_{\infty, w}) \simeq \Ima(H^1(K_{\infty}, \widehat{E}_v[p^\infty]))$ \cite[Proposition 2.4]{Gre99}.
			
			In particular, $H^1(K_{\infty, w}, \tilde{E}_v [p^\infty])[p] \simeq H^1(K_{\infty, w}, E)[p]$ and the surjective map $H^1(K_{\infty, w}, \tilde{E}_v[p]) \rightarrow H^1(K_{\infty, w}, \tilde{E}_v[p^\infty])[p]$ takes values in $H^1(K_{\infty, w}, E)[p]$, with kernel $\tilde{E}_v[p](k_{\infty, w})$ of $\F_p$-dimension at most $1$.
		\end{enumerate}
	\end{proof}

	\begin{thm} \label{mod-p}
		Assume hypotheses \nameref{surj}, \nameref{p-ram} and \nameref{van}. We have an injection
		$$\phi: R(K_\infty, E[p]) \hookrightarrow \Sel(K_\infty, E)[p]$$
		whose cokernel is finite. In particular,
		$$\corank_{\Omega(G)} R(K_\infty, E[p]) = \corank_{\Omega(G)} \Sel(K_\infty, E[p^\infty])[p].$$
	\end{thm}
	\begin{proof}
		Consider the following diagram
		$$\begin{tikzcd}
			0 \arrow[r] & R(K_\infty, E[p]) \arrow[r] \arrow[d, "\phi"] & H^1(K_S/K_\infty, E[p]) \arrow[r, "\xi_p"] \arrow[d, "\simeq"] & \bigoplus_{v \in S} \tilde{K}_v(K_\infty, E[p]) \arrow[d, "\oplus_{v \in S} \phi_v"] \\
			0 \arrow[r] &  \Sel(K_\infty, E[p^\infty])[p] \arrow[r] &  H^1(K_S/K_\infty, E[p^\infty])[p] \arrow[r] & \bigoplus_{v \in S} J_v(K_\infty, E[p^\infty])[p].
		\end{tikzcd}$$
		where each prime $v$ in the direct sum is a prime in $K$. The middle vertical arrow is an isomorphism, because of Theorem \ref{ker}(i) which follows from \nameref{surj}.
		
		Each local arrow $\phi_v$ can be rewritten as $\phi_v = \oplus_{w \mid v} \phi_w$,
		where
		\[\phi_{w}: H^1(K_{\infty, w}, E[p]) \rightarrow H^1(K_{\infty, w}, E[p^\infty])[p]\]
		if $v \in S$ but $v \nmid p$,
		and
		\[\phi_{w}: H^1(K_{\infty, w}, \tilde{E}_v[p]) \rightarrow (H^1(K_{\infty, w}, E[p^\infty])/\Ima(\kappa_{\infty, w}))[p]\]
		if $v \mid p$.
		Applying the snake lemma shows that
		\[\coker(\phi) \subseteq \bigoplus_{w \mid S, w \nmid p} \ker(\phi_w) \oplus \bigoplus_{w  \mid p} \ker(\phi_w),\]
		and the equality holds if and only if $\xi_p$ is surjective. It follows that 
		\[\dim_{\F_p}{\coker(\phi)} \leq \sum_{w \mid S, w \nmid p} \dim_{\F_p} E(K_{\infty, w})[p] + \sum_{w \mid p} \dim_{\F_p} \tilde{E}_v(k_{\infty, w})[p]\]
		by Theorem \ref{ker}. Note that it is not true in general that the sum on the right hand side is finite. However, we assumed that $E(K_{\infty, w})[p] = E(K_v)[p] = 0$ for primes $v$ that are infinitely split in $K_\infty/K$ (hypothesis \nameref{van}). For $K_\infty = K_\ac$, these are the primes that are inert in $K/\Q$.
		
		Once again, the equality holds if and only if $\xi_p$ is surjective.
	\end{proof}
	

\begin{lemma} \label{cork-equal}
	Suppose that hypothesis \nameref{cotor} holds for $E/K$. Then we have
	\begin{enumerate}[(i)]
		\item $\corank_{\Lambda(G)} H^1(K_S/K_\infty, E[p^\infty]) = \sum_{v \mid p} \corank_{\Lambda(G)} J_v^{1}(K_\infty, E[p^\infty]).$ 
		\item $\corank_{\Omega(G)} H^1(K_S/K_\infty, E[p]) = \sum_{v \mid p} \corank_{\Omega(G)} \tilde{K}_v^{1}(K_\infty, E[p]).$
	\end{enumerate}
\end{lemma}

\begin{proof}
	\begin{enumerate}[(i)]
		\item The computations in Theorem \ref{thm:weak-Leop} immediately give \[\corank_{\Lambda(G)} H^1(K_S/K_\infty, E[p^\infty]) = \sum_{v \mid p} \corank_{\Lambda(G)} J_v(K_\infty, E[p^\infty]) = [K:\Q].\] 
		\item This follows from part (i) and Theorem \ref{ker}, following the same argument as \cite[Proposition 4.12]{CS23}.
	\end{enumerate} 
\end{proof}

Consider the composition
\[\pr_{p}: \bigoplus_{v \in S} \tilde{K}_v^{1}(K_\infty, E[p])  \rightarrow \bigoplus_{v \mid p} \tilde{K}_v^{1}(K_\infty, E[p])\]
 and let $\vartheta_{p}$ denote the composition $\pr_p \circ \xi_p.$ 

\begin{thm} Suppose that $E/K$ is an elliptic curve satisfying hypotheses \nameref{ord}, \nameref{surj}, \nameref{cotor}, \nameref{van} and \nameref{p-ram}. \label{equiv}
	The following conditions are equivalent:
	\begin{enumerate}[(i)]
		\item $\vartheta_p$ is surjective and $\mu(Y(K_\infty, E[p^\infty])) = 0$.
		\item $\xi_p$ is surjective and $\mu(Y(K_\infty, E[p^\infty])) = 0$.
		\item $R(K_\infty, E[p])$ is $\Omega(G)$-cotorsion.
		\item $X(K_\infty, E[p^\infty])$ has trivial $\mu$-invariant.
	\end{enumerate}
\end{thm}

\begin{proof}
	We first show that (i) and (ii) are equivalent. Since $\pr_p$ is surjective, (ii) clearly implies (i). Conversely, suppose that $\vartheta_p$ is surjective and consider the commutative diagram
	
	\[
	\begin{tikzcd}
		H^1(K_S/K_\infty, E[p]) \arrow[rd, "\vartheta_p"] \arrow[r, "\xi_p"] &  \bigoplus_{v \in S} \tilde{K}_v^{1}(K_\infty, E[p]) \arrow[d, "\pr_p"] \\
		&  \bigoplus_{v \mid p} \tilde{K}_v^{1}(K_\infty, E[p]).
	\end{tikzcd}
	\]
	
	This induces the exact sequence
	\[\ker(\pr_p) \rightarrow \coker(\xi_p) \rightarrow \coker(\vartheta_p) = 0.\]
	Note that $\ker(\pr_p) = \bigoplus_{v \in S, v \nmid p} \tilde{K}_v^{1}(K_\infty, E[p])$ is $\Omega(G)$-cotorsion (see the Proof of \cite[Proposition 4.12]{CS23}).  Hence $\coker(\xi_p)$ is also $\Omega(G)$-cotorsion. By Theorem \ref{thm:weak-Leop}, we conclude that the weak Leopoldt conjecture holds, i.e. that $H^2(K_S/K_\infty, E[p^\infty]) = 0$. Theorem \ref{main-1} then gives $H^2(K_S/K_\infty, E[p]) = 0$. 
	
	It then follows from the Cassels-Poitou-Tate exact sequence that the following is exact:
	\begin{align*}
		0 \rightarrow R(K_\infty, E[p]) \rightarrow H^1(K_S/K_\infty, E[p]) \xrightarrow{\xi_p} \bigoplus_{v \in S} \tilde{K}_v^{1}(K_\infty, E[p]) \rightarrow \calR(K_\infty, E[p])^{\vee}\rightarrow 0.
	\end{align*} 

	So $\coker(\xi_p) = \calR(K_\infty, E[p])^{\vee}$ is $\Omega(G)$-cotorsion, which is also cofree by Lemma \ref{free}. Thus, $\coker(\xi_p) = 0,$ as required.

	We show that (ii) implies (iii) as follows. Since $\xi_p$ is surjective, the defining sequence for the fine residual Selmer group is exact:
	$$0\rightarrow R(K_\infty, E[p]) \rightarrow H^1(K_S/K_\infty, E[p]) \xrightarrow{\xi_p} \bigoplus_{v \in S} \tilde{K}_v^{1}(K_\infty, E[p]) \rightarrow 0.$$
	By Lemma \ref{cork-equal}, the $\Omega(G)$-coranks of the last two terms are equal. Therefore, $R(K_\infty, E[p])$ must be $\Omega(G)$-cotorsion. 
	Moreover, (iv) and (iii) are equivalent, by Theorem \ref{mod-p} and the fact that $X(K_\infty, E[p^\infty])$ has trivial $\mu$-invariant if and only if $X(K_\infty, E[p^\infty])/ p X(K_\infty, E[p^\infty])$ is finite as a $\Omega(G)$-module.
	
	It remains to show that (iii) implies (ii). Indeed, (iii) implies (iv) which implies that $\mu(Y(K_\infty, E[p^\infty])) = 0$ because $Y(K_\infty, E[p^\infty])$ is a quotient of $X(K_\infty, E[p^\infty])$. Once again, we have the following exact sequence:
	\begin{align*}
		0 \rightarrow R(K_\infty, E[p]) \rightarrow H^1(K_S/K_\infty, E[p]) \xrightarrow{\xi_p} \bigoplus_{v \in S} \tilde{K}_v^{1}(K_\infty, E[p]) \rightarrow \calR(K_\infty, E[p])^{\vee}\rightarrow 0.
	\end{align*} 
	
	Observe that $\corank_{\Omega(G)} H^1(K_S/K_\infty, E[p]) = \corank_{\Omega(G)} \bigoplus_{v \in S} \tilde{K}_v^{1}(K_\infty, E[p]) = 2$ by Lemma \ref{cork-equal} above. Thus, \[\corank_{\Lambda(G)} R(K_\infty, E[p]) = \corank_{\Lambda(G)} \calR(K_\infty, E[p])^{\vee}\] and condition (iii) implies $\calR(K_\infty, E[p])^{\vee}$ is $\Omega(G)$-torsion. Under hypothesis \nameref{cotor}, Lemma \ref{free} implies that $\calR(K_\infty, E[p])$ is free, hence it must follow that $\calR(K_\infty, E[p]) = 0$, as desired. 
\end{proof}

We are now ready to prove the first main result, which is Theorem \ref{mu}.
\begin{proof}[Proof of Theorem \ref{mu}]
	For $j = 1, 2$ let $\vartheta_{p, j}$ be the map 
	\[\vartheta_{p, j}: H^1(K_S/K_\infty, E_j[p]) \rightarrow \bigoplus_{v \mid p} \tilde{K}_v^{1}(K_\infty, E_j[p])\]
	for the residual representation $E_j[p]$.
	By  Theorem \ref{main-1} regarding fine Selmer groups and the equivalence between (i) and (iii) in Theorem \ref{equiv}, it suffices to show that $\vartheta_{p, 1}$ is surjective if and only if $\vartheta_{p, 2}$ is surjective. 
	
	Let $(\tilde{E_1})_v, (\tilde{E_2})_v$ be the reductions modulo $v$ of $E_1, E_2$ at some prime $v \mid p$ in $K$. Fix an isomorphism $E_1[p] \simeq E_2[p]$, then there is an induced isomorphism $(\tilde{E_1})_v[p] \simeq (\tilde{E_2})_v[p]$ such that the diagram
	$$\begin{tikzcd}
		E_1[p] \arrow[r] \arrow[d, "\simeq"] & (\tilde{E_1})_v[p] \arrow[d, "\simeq"] \\
		E_2[p] \arrow[r] & (\tilde{E_2})_v[p] 
	\end{tikzcd}$$
	commutes, by the proof of \cite[Proposition 3.9]{CS23}. The vertical isomorphisms induce the following isomorphisms in local cohomology
	\[H^1(K_{\infty, w}, (\tilde{E}_1)_v[p]) \simeq H^1(K_{\infty, w}, (\tilde{E}_2)_v[p])\]
	for $w \mid p$. This shows that $\vartheta_{p, 1}$ is surjective if and only if $\vartheta_{p, 2}$ is surjective. 
\end{proof}

When the $\mu$-invariant vanishes for both $E_1$ and $E_2$, one can derive a comparison formula for the $\lambda$-invariants. First of all, let us give the following definition.
\begin{defn} \label{expl-const}
	We will denote by $\delta_E$ the quantity appearing in the inequality of Theorem \ref{mod-p}:
	\[\delta_E = \sum_{w \mid S, w \nmid p} \dim_{\F_p} E(K_{\infty, w})[p] + \sum_{w \mid p} \dim_{\F_p} \tilde{E}_v(k_{\infty, w})[p]\]
\end{defn}

\begin{thm} \label{lambda}
	Suppose $E_1/K$, $E_2/K$ are elliptic curves such that $E_1[p] \simeq E_2[p]$ as $\Gal(\overline{K}/K)$-modules and that they satisfy hypotheses \nameref{ord}, \nameref{cotor}, \nameref{surj}, \nameref{p-ram}, \nameref{van} and that \[\mu(\Sel(K_\infty, E_1[p^\infty])) = 0. \] Then $\mu(\Sel(K_\infty, E_2[p^\infty])) = 0$
	and $\lambda(\Sel(K_\infty, {E_j})) = \rho + \delta_{E_j}$, where $\rho$ only depends on the residual representation $E_1[p] \simeq E_2[p]$ and $\delta_{E_j}$ is defined in Definition \ref{expl-const}. 
\end{thm}
\begin{proof}
	It follows directly from Theorem \ref{mu} that $\mu(\Sel(K_\infty, E_2[p^\infty])) = 0.$
	
	It is also known that $X(K_\infty, E_j[p^\infty])$ does not have any finite-index submodule for $j = 1, 2$, by \cite[Proposition 4.14]{Gre99}. Since $X(K_\infty, E_j[p^\infty])$ are also $\Lambda(G)$-torsion for $j = 1, 2$ and 
	\[\mu(\Sel(K_\infty, E_1[p^\infty])) = \mu(\Sel(K_\infty, E_2[p^\infty])) = 0,\] one must have that
	\[\lambda(\Sel(K_\infty, {E_j}[p^\infty])) = \dim_{\F_p} (X(K_\infty, E_j[p^\infty])/ pX(K_\infty, E_j[p^\infty]))\]
	By taking the dual of the inclusion in Theorem \ref{mod-p} for $E_j$, one obtains an exact sequence
	\[V_j \hookrightarrow X(K_\infty, E_j[p^\infty]) / pX(K_\infty, E_j[p^\infty]) \twoheadrightarrow {R(K_\infty, E_j[p])}^{\vee},\]
	where $V_j$ is the dual of the cokernel of $R(K_\infty, E_j[p]) \hookrightarrow \Sel(K_\infty, E_j)[p]$. Since $\mu(\Sel(K_\infty, {E_j})) = \mu(\Sel(K_\infty, {E_j})) = 0$, 
	each $\xi_{p, j}$ in Theorem \ref{mod-p} is surjective by Theorem \ref{equiv} for $j = 1, 2$. Taking the $\F_p$-dimension of every term in the short exact sequence above gives:
	\[\lambda(\Sel(K_\infty, {E_j})) = \dim_{\F_p}R(K_\infty, E_j[p])^{\vee} + \dim_{\F_p}(V_j).\]
	Note that $\dim_{\F_p}(V_j) = \delta_{E_j}$ as defined in \ref{expl-const}. One obtains the required result by letting $\rho = \dim_{\F_p}R(K_\infty, E_j[p])^\vee$.
\end{proof}

In the remainder of this section, we give a formula for comparing $\Lambda(H)$-coranks of Selmer groups over the compositum $F_\infty$ of all $\Z_p$-extensions when $\mu = 0$ over one particular $\Z_p$-extension $K_\infty$, where $H = \Gal(F_\infty/K_\infty)$. An analogous result is obtained in \cite{Lim18} where  $K_\infty$ is the cyclotomic extension over $K$ and $F_\infty$ is a certain $p$-adic Lie extension containing $K_\infty$. However, we do not assume that $K_\infty$ is the cyclotomic extension, but rather a general $\Z_p$-extension.

As before, let $E/K$ be an elliptic curve. Consider the $\Z_p^2$-extension $F_\infty/K$ where 
\[F_\infty = \bigcup_{\Gal(L_\infty/K) \simeq \Z_p} L_\infty.\] When the hypotheses \nameref{ord}, \nameref{cotor}, \nameref{p-ram}, \nameref{van} are satisfied for $E$, there is a morphism
\[\Sel(K_\infty, E[p^\infty]) \rightarrow \Sel(F_\infty, E[p^\infty])^H\]
with finite kernel and cokernel, following the proof of \cite[Proposition 2.8, Proposition 4.1]{KMR23}. In particular, when $\mu(E) = 0$, $\Sel(K_\infty, E[p^\infty])$ is co-finitely generated over $\Lambda(H)$.



Moreover, when $E/\Q$ also satisfies hypothesis \nameref{surj}, we may conclude that $E(F_\infty)[p^\infty] = 0$ following an argument similar to Remark \ref{rmk:surj}.

Then, as remarked on \cite[page 30]{KMR23}, it follows from the proof of \cite[Proposition 4.2]{KMR23} that $\corank_{\Lambda(H)}(X(E/F_\infty)) = \lambda(X(E/K_\infty))$. The following theorem is a straightforward corollary of Theorem \ref{lambda}:

\begin{thm} \label{thm:Z_p^2}
	Suppose $E_1/K$, $E_2/K$ are elliptic curves such that $E_1[p] \simeq E_2[p]$ and the hypotheses \nameref{ord}, \nameref{cotor}, \nameref{surj}, \nameref{p-ram}, \nameref{van} hold for both $E_1$ and $E_2$. Further suppose that $\mu(\Sel(K_\infty, E_1[p^\infty])) = 0$. Then $\mu(\Sel(K_\infty, E_2[p^\infty])) = 0$,
	and $$\corank_{\Lambda(H)}(\Sel(F_\infty, {E_j})) = \rho + \delta_{E_j},$$ where $\rho$ only depends on the residual representation $E_1[p] \simeq E_2[p]$ and $\delta_{E_j}$ is in terms of local constants as defined in Definition \ref{expl-const}.  
\end{thm}

\section{Selmer groups over the False-Tate extension} \label{False-Tate}

We consider the False-Tate extension $F_\infty = K(\mu_{p^\infty}, m^{1/p^{\infty}})$ over $K = \Q(\mu_p)$ for some $p^{th}$-power free integer $m$. Following \cite[Appendix A]{DDCS07}, we define an elliptic curve $E/\Q$ to be regular over the False-Tate extension $F_\infty$ if $\Sel(F_\infty, E[p^\infty]) = 0$. For conditions under which this occurs, we refer readers to \cite[Corollary A.11]{DDCS07}. We will study this property with respect to congruent elliptic curves in Section \ref{sec:reg}. Moreover, $\Sel(F_\infty, E[p^\infty])$ can acquire positive $\Lambda(H)$-corank where $H = \Gal(F_\infty/K_\cyc) \simeq \Z_p$ \cite[Theorem A.32]{DDCS07}. When the $\Lambda(H)$-corank is $1$, there is an explicit formula for the growth of the Selmer groups over the finite layers of $F_\infty$ \cite[Theorem 4.8]{CFKS10}. This will be studied in Section \ref{sec:rk-1}. As before, we denote by $K_{\cyc}$ be the cyclotomic $\Z_p$-extension of $K$, i.e. $K_{\cyc} = \Q(\mu_{p^\infty})$. We shall also denote the finite layers of $K_\cyc$ by $K_n = \Q(\mu_{p^n})$. 

We recall the definitions of imprimitive Selmer groups and imprimitive Euler characteristics. For an elliptic curve $E/\Q$, let $S$ be a finite set of primes of $K$ containing the primes lying above $p$, the infinite primes and the primes of bad reduction of $E$. Denote by $K_S$ the maximal extension of $K$ unramified outside of $S$. Let $\Sigma_0 \subset S$ be a set of primes not containing the primes above $p$. For a Galois extension $L/K$ where $K \subset L \subset K_S$, we define the  $\Sigma_0$-imprimitive Selmer group  $\Sel^{\Sigma_0}(L, E[p^\infty])$ as the kernel of 
\[H^1(K_S/L, E[p^\infty]) \rightarrow \bigoplus_{v \mid S \setminus \Sigma_0} \bigoplus_{w \mid v} H^1(L_w, E[p^\infty])/\Ima(\kappa_w),\]
where $\kappa_w: E(L_w) \otimes \Q_p/\Z_p \rightarrow H^1(L_w, E[p^\infty])$ is the local Kummer map. In the special case where $\Sigma_0 = \emptyset$, one recovers Definition \ref{def:Sel}.
We will denote by $X^{\Sigma_0}(L, E[p^\infty])$ the Pontryagin dual of $\Sel^{\Sigma_0}(L, E[p^\infty])$ and denote $\chi(L/K, E; \Sigma_0)$ to be the Euler characteristic $\chi(\Gal(L/K), X^{\Sigma_0}(L, E[p^\infty]))$ defined in Section \ref{def:Euler-char}. This is called the imprimitive Euler characteristic. When $\Sigma_0 = \emptyset$, we simply denote $\chi(L/K, E)$ for $\chi(L/K, E; \Sigma_0).$ 

Throughout this section, $E_1/\Q, E_2/\Q$ will denote elliptic curves of respective conductors $N_1, N_2$ that are both ordinary at a fixed odd prime $p$ with isomorphic irreducible mod-$p$ representations $E_1[p] \simeq E_2[p]$. We will call $\overline{\rho}: G_\Q \rightarrow GL_2(\F_p)$ this mod-$p$ representation and denote by  $\overline{N}$ its conductor, following the notations of \cite{SS14}. In this section, let $K = \Q(\mu_p)$ and $v$ denote a non-archimedean prime in $K$. For such a pair of elliptic curves, we define the following hypothesis, which will be used in theorems \ref{cong:FT2}, \ref{cong:cyc} below.
\begin{hypo} \label{A}
	$v \mid \overline{N}$ for all primes $v$ where $E_2$ has split multiplicative reduction. In addition, for $p = 3$, $v \nmid \overline{N}$ or $c_v \in \Z_3^{\times}$ for all $v \mid 3$ where $E_2$ has additive reduction and $c_v$ is the local Tamagawa number of $E_2$ at $v$.		
\end{hypo}

\subsection{The regular case} \label{sec:reg}
Let us first examine the regular case. Consider an elliptic curve $E/\Q$ of conductor $N$ and the False-Tate extension $F_\infty = K_{\cyc}(m^{1/p^{\infty}})$, where $m$ is coprime with $pN$. We follow the notations in \cite{SS14}:

For a $p$-ordinary elliptic curve $E/\Q$ of conductor coprime with $m$, we let $\Sigma_3(E)$ be the set of primes $v$ of $K$ such that $v \mid m$, $v \nmid p$, $E$ has good reduction at $v$ and the corresponding reduced curve has a point of order $p$. This is the same as the set denoted $P_2^{(K)}$ mentioned in \cite[Section 4]{DDCS07}. We recall the following theorem from \cite{SS14}, which is the $\Sigma_0$-imprimitive extension of \cite[Theorem 4.10]{HV03}. Note that \cite[Theorem 4.10]{HV03} is stated for $p \geq 5$, but as remarked in the proof of \cite[Cor 2.8]{SS14}, such a result also holds for $p = 3$. 

\begin{thm} \cite[Cor 2.8]{SS14} \label{FT-cyc}
	If $\chi\left(F_{\infty} / K, E ; \Sigma_0\right)$ exists, so does $\chi\left(K_{\mathrm{cyc}} / K, E ; \Sigma_0\right)$ and we have that
	\[\chi\left(F_{\infty} / K, E ; \Sigma_0\right)=\chi\left(K_{\mathrm{cyc}} / K, E ; \Sigma_0\right) \times \prod_{v \in \Sigma_3(E)}\left|L_v(E, 1)\right|_p,
	\]
	where each $L_v(E, s)$ is the local $L$-factor of $E$ at $v$, as defined in \cite[p. 330]{SS14}, \cite[Remark 4.6]{DDCS07}.
\end{thm}

\begin{rmk} \label{L-integrality}
	We remark that $L_v^{-1}(E, 1)$ is $p$-adically integral for each $v$. Indeed, $L_v^{-1}(E, 1)$ is given by $(Nv)^{-1} \# E(\F_v)$  where $Nv$ is the number of elements of $\F_v$ \cite[Remark 4.6]{DDCS07}, which is a $p$-adic unit since $v \nmid p$.
\end{rmk}

When $\mu(X(E/\Kcyc)) = 0$, it is known that $\chi(F_{\infty}/ K, E) = 1$ is equivalent to $E$ being regular over $F_\infty$  \cite[Prop A.9]{DDCS07}, i.e. that $X(F_\infty, E[p^\infty]) = 0$. As in \cite{SS14}, we give the following definition:

\begin{defn} \cite[Def 3.1]{SS14} \label{def:Sigma}
	Let $\Sigma\left(E_j\right)$ denote the set of primes $v$ such that $v \nmid \bar{N}$ and $E_j$ has split multiplicative reduction at $v$ for $j=1,2$. In the case $p=3$, we shall further assume that $\Sigma\left(E_j\right)$ also contains the finite primes $v$ of $K$ which satisfy the following conditions: (i) $v \mid N_j / \bar{N}$, (ii) $E_j$ has additive reduction at $v$ and (iii) the local Tamagawa number $c_v$ is not a 3 -adic unit.
\end{defn} 

We will leverage the following result of Sujatha-Shekhar to study the mod-$p$ invariance of regularity over $F_\infty$. 

\begin{thm} \label{cong:FT} \cite[Thm 3.4]{SS14} 
	Let $E_1/\Q, E_2/\Q$ be elliptic curves of conductors such that $E_1[p] \simeq E_2[p]$ as $G_\Q$-modules. Suppose that $F_\infty = K_\cyc(m^{1/p^\infty})$ is a False Tate extension where $m$ is coprime with $pN_1N_2$ where $N_1$ and $N_2$ are respectively the condutors of $E_1$ and $E_2$. Suppose that:
	\begin{enumerate}[(i)]
		\item $E_1(K)[p] = 0$.
		\item $\Sigma(E_2) \subset \Sigma_0 \subset \Sigma(E_1) \cup \Sigma(E_2)$.
		\item $\chi(F_\infty/K, E_1; \Sigma_0) = 1.$
	\end{enumerate}
	Then $\Sigma_0 = \emptyset$ and $\chi(F_\infty/K, E_2; \Sigma_0) = \chi(F_\infty/K, E_1) = \chi(F_\infty/K, E_2) = 1.$
\end{thm}

We remark that Hypothesis \ref{A} is equivalent to imposing that $\Sigma(E_2) = \emptyset$. The following theorem is obtained by taking $\Sigma_0 = \emptyset$ in the hypotheses of Theorem \ref{cong:FT}.

\begin{thm} \label{cong:FT2}
	Let $E_1/\Q, E_2/\Q$ be elliptic curves such that $E_1[p] \simeq E_2[p]$ as $G_\Q$-modules. Suppose that $F_\infty = K_\cyc(m^{1/p^\infty})$ is a False Tate extension where $m$ is coprime with $pN_1N_2$ where $N_1$ and $N_2$ are respectively the condutors of $E_1$ and $E_2$. Assume that Hypothesis \ref{A} holds for $E_2$, i.e. $\Sigma(E_2) = \emptyset$. In particular, this means:
	\begin{enumerate}[(i)]
		\item $v \mid \overline{N}$ for all $v$ where $E_2$ has split multiplication reduction.
		\item In addition, for $p = 3$, $v \nmid \overline{N}$ or $c_v \in \Z_3^{\times}$ for all $v \mid 3$ where $E_2$ has additive reduction.
	\end{enumerate} 
	Further assume that $E_1(K)[p] = 0$ and that $E_1$ is regular over $F_\infty$, i.e. $X(F_\infty, E_1[p^\infty]) = 0$. Then one also has $X(F_\infty, E_2[p^\infty]) = 0$.
\end{thm}

\begin{proof}
	It is clear that $X(F_\infty, E_1[p^\infty]) = 0$ implies $\chi(F_\infty/K, E_1) = 1$. We claim that $\chi(F_\infty/K, E_1) = 1$ implies  $\chi (\Kcyc/K, E_1) = 1$. Indeed, by Remark \ref{L-integrality}, $|L_v(E, 1)|_p$ must be a non-negative power of $p$. Because $\chi(\Kcyc/K, E_1)$ is $p$-adically integral (\cite[Lemma 4.2]{Gre99}), it follows from Theorem $\ref{FT-cyc}$ that $\chi(\Kcyc/K, E_1) = 1$. This implies that $\Sigma_3(E_1) = \emptyset$ and  $X(\Kcyc, E_1[p^\infty]) = 0$. In particular, $\mu(X(\Kcyc, E_1[p^\infty])) = 0$. By the work of Greenberg-Vatsal \cite{GV00}, we also have $\mu(X(\Kcyc, E_2[p^\infty])) = 0$ and thus there is a equivalence $X(F_\infty, E_2[p^\infty]) = 0 \Longleftrightarrow \chi(F_\infty/K, E_2) = 1$, following \cite[Prop A.9]{DDCS07}. 
	
	Because the hypotheses imply that $\Sigma(E_2) = \emptyset$, one may take $\Sigma_0 = \emptyset$ in Theorem \ref{cong:FT} to conclude that $\chi(F_\infty/K, E_2) = 1$, or equivalently $X(F_\infty, E_2[p^\infty]) = 0.$
\end{proof}

The following congruent elliptic curves satisfy the hypotheses of Theorem \ref{cong:FT2}, and are also examined in \cite{SS14}. They will also be used as the main example for some other results to be mentioned in this paper.

\begin{ex}
	\label{cong:ex}
	For $p = 3$, consider the pair of $p$-congruent elliptic curves with Cremona labels $E_1 = 11A3$ \cite[\href{https://www.lmfdb.org/EllipticCurve/Q/11.a3}{Elliptic Curve 11.a3}]{lmfdb} and $E_2 = 77C1$ \cite[\href{https://www.lmfdb.org/EllipticCurve/Q/77.c2}{Elliptic Curve 77.c2}]{lmfdb}. They both satisfy the hypothesis in Theorem \ref{cong:FT2} because  $E_2$ has split multiplicative reduction at $11$, non-split multiplicative reduction at $7$ and $11 \mid \overline{N}$. Indeed, since $E_1$ is a Tate curve over $\Q_{11}$, one has the following isomorphism of $\Gal(\overline{\Q_{11}}/ \Q_{11})$-modules:
	
	$$\overline{\Q_{11}}^\times / q_{\text{Tate}}^{\Z} \longrightarrow E_1(\overline{\Q_{11}})$$ 
	
	Hence $E_1(\overline{\Q_{11}})[3] \simeq (\overline{\Q_{11}}^\times/q_{\Tate})[3] = \{\alpha \in \overline{\Q_{11}}^\times \mid \alpha^3 = q_{\text{Tate}}^k \text{ for some } k \in \Z\}$ as Galois modules. But one can verify computationally that $3 \nmid v_{11}(q_{\text{Tate}})$. Hence the representation attached to $E_1[3]$ is ramified at $11$. Moreover, neither $E_1$ nor $E_2$ has primes of additive reduction above $3$.

	Since $X(F_\infty, E_1[p^\infty]) = 0$ by the computations in \cite[Cor A.11, Table 3-11A3]{DDCS07}, one also has $X(F_{\infty}, E_2[p^\infty]) = 0$ by Theorem \ref{cong:FT2} above. One can check that this is consistent with \cite[Cor A.11, Table 3-77C1]{DDCS07}.	
\end{ex}

\begin{rmk}
	One can generalize the example above with any mod-$p$ congruent curves $E_1$ and $E_2$ where:
	\begin{itemize}
		\item $E_1$ only admits bad reduction at a split-multiplicative prime.
		\item $E_1(K)[p] = 0$.
		\item $E_2$ admits split multiplicative reduction only at primes dividing $\overline{N}$.
		\item If $p = 3$, then we additionally require $v \nmid \overline{N}$ or $c_v \in \Z_3^\times$ for all $v$ where $E_2$ has additive reduction.
	\end{itemize} 
\end{rmk}

The following result follows from the proof of \cite[Theorem 3.4]{SS14}.

\begin{thm} \label{cong:cyc}
	Let $E_1/\Q, E_2/\Q$ be elliptic curves such that $E_1[p] \simeq E_2[p]$ as $G_\Q$-modules. Suppose that:
	\begin{enumerate}[(i)]
		\item $E_1(K)[p] = 0$.
		\item $\Sigma(E_2) \subset \Sigma_0 \subset \Sigma(E_1) \cup \Sigma(E_2)$.
		\item $\chi(K_\cyc/K, E_1; \Sigma_0) = 1$, or equivalently $X^{\Sigma_0}(K_\cyc, E_1[p^\infty]) = 0$.
	\end{enumerate}
	Then we also have $\chi(K_\cyc/K, E_2; \Sigma_0) = 1$, or equivalently $X^{\Sigma_0}(K_\cyc/K, E_2[p^\infty]) = 0.$
\end{thm}

We remark that mposing Hypothesis \ref{A} is the same as imposing that $\Sigma_0 = \emptyset$ in Theorem \ref{cong:cyc}. In particular, one has

\begin{thm} \label{cong:cyc2}
	Let $E_1/\Q, E_2/\Q$ be elliptic curves such that $E_1[p] \simeq E_2[p]$ as $G_\Q$-modules, and assume that Hypothesis \ref{A} holds for $E_2$, i.e. $\Sigma(E_2) = \emptyset$. In particular, this means:
	\begin{enumerate}[(i)]
		\item $v \mid \overline{N}$ for all $v$ where $E_2$ has split multiplication reduction.
		\item In addition, for $p = 3$, $v \nmid \overline{N}$ or $c_v \in \Z_3^{\times}$ for all $v \mid 3$ where $E_2$ has additive reduction.
	\end{enumerate} 
	Suppose that $E_1(K)[p] = 0$ and $X(K_\cyc, E_1[p^\infty]) = 0$. Then one also has $X(K_\cyc, E_2[p^\infty]) = 0$.
\end{thm}

Theorem \ref{cong:cyc2} applies to Example \ref{cong:ex} where $p = 3$: $E_1(K)[p] = 0$ because $K/\Q$ is an imaginary quadratic extension, and $\Q(E_1[p])/\Q$ is a $GL_2(\Z/p\Z)$-extension. Moreover, $X(\Kcyc, E_1[p^\infty]) = 0$ because of the computations in \cite[Table3-11A3]{DDCS07} as well as the argument in the proof of \cite[Corollary A.11]{DDCS07}.

	%



\subsection{The rank $1$ case} \label{sec:rk-1}

Keeping the same notations as the previous subsection, we study the False-Tate extension $F_\infty = \cup_{n \geq 0} F_n$ where $F_n = \Q(\mu_{p^n}, m^{1/p^n})$. Let $L_\infty = \cup_{n \geq 0} L_n$ where $L_n = \Q(x_n)$ for some compatible choices of $x_n = m^{1/p^n}$, i.e. $x_{n + 1}^{p} = x_n$. As before, let $H = \Gal(F_\infty/K_\cyc)$, which is isomorphic to $\Z_p$.

For an elliptic curve $E/\Q$, recall from the previous section that $\Sigma_3(E)$ is the set of primes $v$ in $K$ such that $v \mid  m, v \nmid p$, $E$ has good ordinary reduction at $v$ and $p \mid \# \tilde{E}_q(k_v)$. 

First of all, we give the following simplification of \cite[Proposition 4.7]{CFKS10} regarding the $\Lambda(H)$-rank of $X(F_\infty, E[p^\infty])$. We also remark that a more precise statement of \cite[Proposition 4.7]{CFKS10} can be found in \cite[Theorem A.32, 3]{DDCS07}, where there is an explicit formula for the $\Lambda(H)$-rank of $X(F_\infty, E[p^\infty])$, and an exhaustive description of cases where the $\Lambda(H)$-rank is exactly $1$.

\begin{thm} \label{FT-rk}
	Assume that $E$ has good ordinary reduction at $p$ and that $\mu(X(K_\cyc, E[p^\infty])) = 0$. Then a necessary condition for $X(F_\infty, E[p^\infty])$ to have $\Lambda(H)$-rank $1$ is that $\Sigma_3(E) = \emptyset$. This condition is also sufficient if we additionally assume that $\lambda(E) = 0$ and, among the primes dividing $m$, $E$ has split multplicative reduction at a unique prime $q$ where $q$ is inert in $K$.
\end{thm}

In our setting, one can formulate the $\frakM_H(G)$ conjecture (\cite[p.21]{CFKS10}), namely that \[X(F_\infty, E[p^\infty])/X(F_\infty, E[p^\infty])[p^\infty]\] is finitely generated over $\Lambda(H)$ as a $\Lambda(G)$-module, where $G = \Gal(F_\infty/K)$. In particular, this condition is satisfied whenever the $\mu$-invariant of $E$ vanishes over $K^{cyc}$, and we are content with this assumption.

We now study the following result due to Coates {\it et al.} \cite{CFKS10} over the False-Tate extension:

\begin{thm} \label{thm:Coates} \cite[Theorem 4.8]{CFKS10}
	Assume that $E/\Q$ is an elliptic curve which is ordinary at $p$ and that $X(K^{\cyc}, E[p^\infty])$ has trivial $\mu$-invariant. Further suppose that $X(F_\infty, E[p^\infty])$ has $\Lambda(H)$-rank $1$. Then for all $n \geq 1$, we have
	\[\rk_{\Z_p}{X(L_n, E_i[p^\infty])} = n + \rk_{\Z_p} X(\Q, E_i[p^\infty]),\]
	\[\rk_{\Z_p}{X(F_n, E_i[p^\infty])} = p^n - 1 + \rk_{\Z_p} X(K, E_i[p^\infty]).\]
\end{thm}

In the following result, we prove that Theorem \ref{thm:Coates} holds in a family of congruent elliptic curves under certain hypotheses:

\begin{thm} \label{cong:growth}
	Assume that $E_1/\Q, E_2/\Q$ are elliptic curves with good ordinary reduction at $p$ such that $E_1[p] \simeq E_2[p]$. Suppose that $X(K_\cyc, E_1[p^\infty]) = 0$ and $E_1$ has split multiplicative reduction at a unique prime $q$ such that $q$ is inert in $K$. Suppose that the following hypotheses of Theorem \ref{cong:cyc2} also hold:
	\begin{enumerate}[(i)]
		\item $E_1(K)[p] = 0$.
		\item $v \mid \overline{N}$ for all $v$ where $E_2$ has split multiplication reduction.
		\item In addition, for $p = 3$, $v \nmid \overline{N}$ or $c_v \in \Z_3^{\times}$ for all $v \mid 3$ where $E_2$ has additive reduction.
	\end{enumerate}
	
	Then for both $i = 1$ and $i = 2$, we have $\mu(E_i) = \lambda(E_i) = 0$, $X(F_\infty, E_i[p^\infty])$ has $\Lambda(H)$-rank $1$ for $F_\infty = \Q(\mu_{p^\infty}, q^{1/p^\infty})$. Furthermore, for $n \geq 1$, one has the following:
	$$\rk_{\Z_p}{X(L_n, E_i[p^\infty])} = n + \rk_{\Z_p} X(\Q, E_i[p^\infty]),$$
	$$\rk_{\Z_p}{X(F_n, E_i[p^\infty])} = p^n - 1 + \rk_{\Z_p} X(K, E_i[p^\infty]).$$
\end{thm}

\begin{proof}
	Suppose that $X(K_\cyc, E_1[p^\infty]) = 0$. It follows from Theorem \ref{cong:cyc2} that we also have $\mu(E_2) = \lambda(E_2) = 0$ for any such elliptic curve $E_2/\Q$ satisfying the hypotheses. Moreover, since $E_2$ does not have split multiplicative reduction at any prime $v \mid N_2$ where $v \nmid N_1$, the hypotheses of Theorem \ref{FT-rk} applies to both $E_1$ and $E_2$. Therefore, $X(F_\infty, E_i[p^\infty])$ has $\Lambda(H)$-rank $1$ for each $i = 1, 2$. 		
	The systematic $\Z_p$-rank growth for both $E_1$ and $E_2$ along the $L_\infty/\Q$ as well as the $F_\infty/K$ tower follows from Theorem \ref{thm:Coates}.
\end{proof}

\printbibliography

\end{document}